\documentclass{amsart}

\usepackage{amsfonts, amssymb}
\usepackage{graphicx}
\usepackage{hyperref}
\usepackage{xcolor}

\newtheorem{thm}{Theorem}[section]
\newtheorem{cor}[thm]{Corollary}

\newtheorem{prop}[thm]{Proposition}
\theoremstyle{definition}
\newtheorem{defn}[thm]{Definition}
\theoremstyle{remark}

\numberwithin{equation}{section}
\newcommand{\D}{\mathcal D}
\newcommand{\R}{\mathbb R}
\newcommand{\N}{\mathbb N}

\newcommand{\F}{\mathcal{F}}

\usepackage{color}
\definecolor{darkblue}{rgb}{0.05, .05, .65}

\begin{document}
\title[]{Fractional Lane-Emden Hamiltonian systems}
	
\author[I. Ceresa Dussel, J. Fern\'andez Bonder, N. Saintier and A. Salort]{Ignacio Ceresa Dussel, Juli\'an Fern\'andez Bonder, Nicolas Saintier and Ariel  Salort}
	
\address{Instituto de C\'alculo, CONICET\\
Departamento de Matem\'atica, FCEN - Universidad de Buenos Aires\\
Ciudad Universitaria, 0+$\infty$ building, C1428EGA, Av. Cantilo s/n\\
Buenos Aires, Argentina}

\email[ICD]{iceresad@dm.uba.ar}
\email[JFB]{jfbonder@dm.uba.ar}
\email[NS]{nsaintier@dm.uba.ar}
\email[AS]{asalort@dm.uba.ar}
	
\begin{abstract}
In this work, our interest lies in proving the existence of solutions to the following Fractional Lane-Emden Hamiltonian system:
$$
\begin{cases}
(-\Delta)^s u = H_v(x,u,v) & \text{in }\Omega,\\
(-\Delta)^s v = H_u(x,u,v) & \text{in }\Omega,\\
u=v=0 & \text{in } \R^n\setminus\Omega.
\end{cases}
$$
The method, that can be traced back to the work of De Figueiredo and Felmer \cite{DF-F}, is flexible enough to deal with more general nonlocal operators and make use of a combination of fractional order Sobolev spaces together with functional calculus for self-adjoint operators.
\end{abstract}
	
	
%35J50  	Variational methods for elliptic systems
%35J47  	Second-order elliptic systems
%35D30  	Weak solutions to PDEs
%35R11  	Fractional partial differential equations
	

\maketitle

\section{Introduction}

Given a bounded and smooth domain $\Omega\subset \R^n$, $n\geq 3$ and suitable functions $f$ and $g$,  existence and non-existence of solutions to Hamiltonian systems of the form
\begin{align} \label{LE5'}
\begin{cases}
-\Delta u = f(v) &\text{ in } \Omega\\
-\Delta v = g(u) &\text{ in } \Omega\\
u=v=0 &\text{ on } \partial\Omega,
\end{cases}
\end{align}
was studied in \cite{Clement-DFigeiredo-Mitidieri, DF-F,HV93,M93}. These articles provide for a  suitable framework to deal with this kind of systems with strongly indefinite variational structure. The prototypical case of \eqref{LE5'} was studied in \cite{Clement-DFigeiredo-Mitidieri} and it is given when $f(v)=v^{p-1}$ and $g(u)=u^{q-1}$, where  $u$ and $v$ are positive functions:
\begin{align} \label{LE5''}
\begin{cases}
-\Delta u = v^{p-1} &\text{ in } \Omega\\
-\Delta v = u^{q-1} &\text{ in } \Omega\\
u=v=0 &\text{ on } \partial\Omega.
\end{cases}
\end{align}
The authors proved existence of at least positive solutions of this system when 
\begin{align} \label{CondExp}
\frac{n-2}{n}<\frac{1}{p}+\frac{1}{q}<1.
\end{align}
On the contrary, when this condition is not fulfilled,  it was proved in \cite{M93} that \eqref{LE5''} does not admit positive solutions if the domain $\Omega$ is smooth and star-shaped. Observe that this problem  can be thought of as a natural extension of the Lane-Emden equation for $u\geq 0$:
\begin{equation} \label{LE}
-\Delta u = u^{p-1} \text{ in }\Omega, \quad u=0 \text{ in } \partial\Omega,
\end{equation}
which (among many variations) was widely studied and it is well known that its solvability is strongly related with the Sobolev inequality.  Indeed it follows from the classical Pohozaev identity  that \eqref{LE} admits a positive solution if and only if $p<2^*=\frac{2n}{n-2}$.

Solutions to \eqref{LE5'} are obtained as critical points of the Lagrangian
\begin{align} \label{FunctionalLocal}
\mathcal{L}(u,v)=\int_\Omega \nabla u\cdot \nabla v\,dx -\int_\Omega F(v)\,dx-\int_\Omega G(u)\,dx
\end{align}
where $F$ and $G$ are  primitives of $f$ and $g$. In contrast with gradient systems, the functional $\mathcal{A}(u,v)=\int_\Omega \nabla u\cdot\nabla v \,dx$, i.e. the quadratic part of $\mathcal{L}$,  is strongly indefinite, in the sense that $(0,0)$ is a saddle point for $\mathcal{A}$ with infinite Morse index. Therefore any decomposition of the solution space into two subspace $H_1$ and $H_2$ such that $\mathcal{A}$ restricted to $H_1$ has a minimum in $(0,0)$, and $\mathcal{A}$ restricted to $H_2$ has a maximum in $(0,0)$, imply that $H_1$ and $H_2$ are infinite dimensional. Hence neither Ambrosetti-Rabinowitz' Mountain Pass Theorem \cite{AR} (which typically assumes a minimum at the origin) nor Rabinowitz' Saddle Point Theorem \cite{R78} (which requires one of the subspaces $H_1$ or $H_2$ to be finite dimensional), can be applied in this case.

To deal with this kind of problems involving strongly indefinite functionals,  some generalizations of the mountain pass and saddle point theorems were developed. We mention for instance the abstract theorems in \cite{BR}, \cite{H83} and \cite{F}.  By means of the topological min-max  approach developed in the above mentioned articles, existence and regularity results for solutions to \eqref{LE5''} were obtained in \cite{Clement-DFigeiredo-Mitidieri, DF-F, HV93, M93}.

Later, more general Hamiltonian system of the  form
\begin{align} \label{LE5}
\begin{cases}
-\Delta u = H_v(x,u,v) &\text{ in } \Omega\\
-\Delta v = H_u(x,u,v) &\text{ in } \Omega\\
u=v=0 &\text{ on } \partial\Omega
\end{cases}
\end{align}
were studied in \cite{DF-F}, where $H$ includes in particular the prototypical Hamiltonian $H(x,u,v)=|u|^p + |v|^q$ corresponding to \eqref{LE5''} with exponents $p, q$ satisfying \eqref{CondExp}. In this case, solutions were obtained in \cite{DF-F} by applying a variational approach using a generalized Mountain Pass Theorem stated in \cite{F} in the context of Hamiltonian systems, which allows to deal with the strongly indefinite structure of the energy functional related to \eqref{LE5}.

On the other hand, the last years experimented an  increasing attention for non-nocal problems due to their many applications to models of diverse branches of science.  One of the most studied nonlocal operator is the well-known fractional Laplacian defined for $s\in (0,1)$ by
$$
(-\Delta)^s u(x)=\text{p.v.  } C(n,s) \int_{\R^n}\frac{u(x)-u(y)}{|x-y|^{n+2s}}\,dy,
$$
where the constant $C(n,s)$ is positive and depends only on $n$ and $s$.  From now on we will omit for ease of notation.  A non local variant of the  classical Lane-Emden system  \eqref{LE5''}, namely 
\begin{align} \label{LE3s}
\begin{cases}
(-\Delta u)^s = v^{p-1} &\text{ in } \Omega,\\
(-\Delta v)^s = u^{q-1} &\text{ in } \Omega,\\
u=v=0 &\text{ in } \mathbb{R}^n\backslash\Omega,
\end{cases}
\end{align}
where $\Omega\subset \R^n$ is a bounded smooth domain in $\R^n$ and  $n> 2s$, was  studied in \cite{Leite-Montenegro} in the context of positive viscosity solutions. Solutions to \eqref{LE3s} are obtained as critical points of Lagrangian 
\begin{equation} \label{func1}
\mathcal{L}(u,v)= \frac12 \int_{\R^n}\int_{\R^n}\frac{ (u(x)-u(y))(v(x)-v(y))}{|x-y|^{n+2s}}\,dxdy  - \frac{1}{p}\int_\Omega v^p\,dx - \frac{1}{q}\int_\Omega u^q \,dx, 
\end{equation}
which, as in the local case \eqref{FunctionalLocal}, turns out to be strongly indefinite.  The authors  in \cite{Leite-Montenegro} proved that when $p,q>1$ are such that
\begin{equation} \label{CondExpNL}
\frac{n-2s}{n}<\frac{1}{p}+\frac{1}{q} <1
\end{equation}
then \eqref{LE3s} admits at least one positive viscosity solution, and no positive viscosity solution is admitted when the previous condition does not hold and the domain is star-shaped.  Observe that assumption \eqref{CondExpNL} simplifies to \eqref{CondExp} when $s=1$.  Moreover \eqref{LE3s} can be seen as a generalization of the fractional Lane-Emden problem
$$
(-\Delta)^s u=u^{p-1} \text{ in } \Omega, \quad u=0 \text{ in } \R^n \setminus \Omega
$$
which by standard variational methods admits a positive energy weak solution if $1<p<2n/(n-2s)$, ($p\neq 2$). 

In order to prove existence of positive solutions to \eqref{LE3s}, the idea  in \cite{Leite-Montenegro}  is to observe that $v=((-\Delta u)^s)^\frac{1}{p-1}$. So, plugging this expression into the second equation in \eqref{LE3s},  reduces the problem to find critical points of a functional depending only on $u$ instead of the functional \eqref{func1}. This method however does not seem to be adaptable to treat more general Lane-Emden system with Hamiltonian $H$.  

Inspired by \cite{DF-F} we consider in this work the following nonlocal version of \eqref{LE5}
\begin{equation}\label{P}
\begin{cases}
(-\Delta)^s u = H_v(x,u,v) & \text{in }\Omega,\\
(-\Delta)^s v = H_u(x,u,v) & \text{in }\Omega,\\
u=v=0 & \text{in } \R^n\setminus\Omega,
\end{cases}
\end{equation}
where $\Omega\subset \R^n$ is an open bounded set with smooth boundary and the so-called Hamiltonian function $H\colon \Omega\times\R^2\to \R$ is of class $C^1$ and satisfies some growths assumptions stated below. 

We remark  that Hamiltonian systems of the form \eqref{P} involving fractional Laplacians of different order, let us say, $s,t\in(0,1)$, lead to non-variational related energy functionals. As far as we know, the techniques introduced in the literature do not seem to be applicable, so we restrict ourselves to the case of two fractional Laplacians with the same order $s$.

The main goal of this article will be to obtain solutions to \eqref{P}. We want to emphasize  that our approach strongly relies on the ideas of \cite{DF-F}.  More precisely, solutions of \eqref{P} will be found as critical points of an energy functional similar to \eqref{func1} which, as was observed in the local case, is strongly indefinite. In contrast with the prototypical system \eqref{LE3s} considered in \cite{Leite-Montenegro}, it seems difficult in general to reduce \eqref{P}  to a functional depending solely on $u$. Instead, we will set up  a suitable functional framework to obtain solutions by means of the generalized Mountain Pass Theorem obtained in \cite{F}.

We introduce now our assumptions on the Hamiltonian system \eqref{P}.  First let us fix some $s\in (0,1)$  and real numbers $p,q>1$ satisfying  \eqref{CondExpNL} namely 
\begin{align} \label{condpq.0} \tag{$pq_0$}
1-\frac{2s}{n} <\frac{1}{p}+\frac{1}{q}<1.
\end{align}
Furthermore, if $n>4s$ we also impose 
\begin{equation}\label{condpq.1} \tag{$pq_1$}
\frac{1}{p}> \frac{n-4s}{2n} \quad \text{ and }\quad  \frac{1}{q}> \frac{n-4s}{2n}.
\end{equation}
We show in Figure \ref{fig:imag1} the region of $(p,q)$ satisfying these assumptions.  

\begin{figure}
\centering
\includegraphics[width=0.9\linewidth]{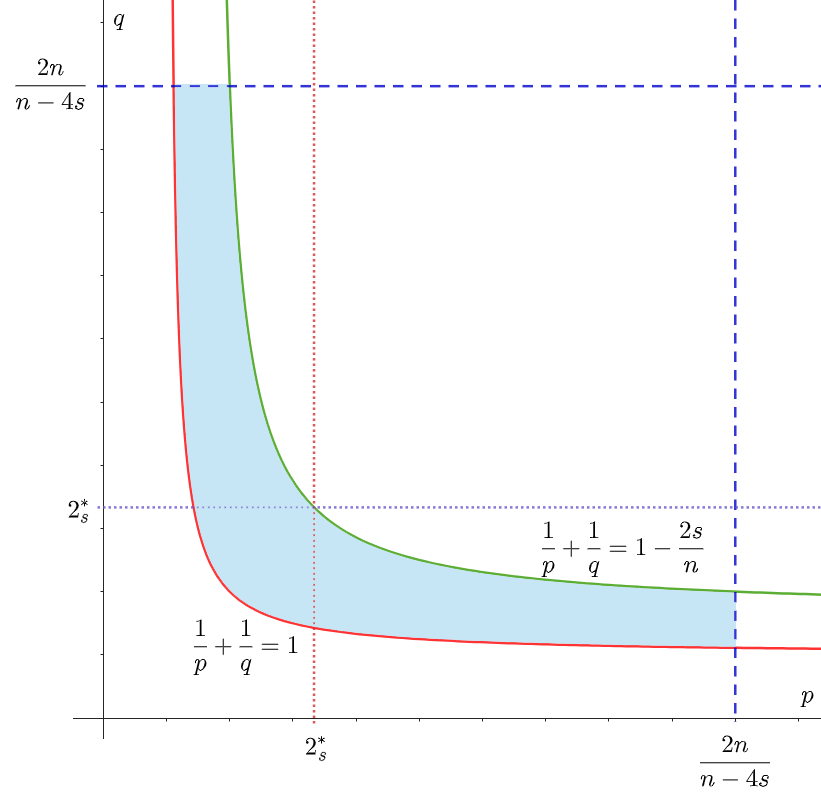}
\caption{Range of admissible $p$ and $q$ in terms of $s$ and $n$.}
\label{fig:imag1}
\end{figure}

We further assume the following conditions on the Hamiltonian:
\begin{align} \label{H0} \tag{$H_0$}
H\colon  \Omega \times \R^2 \to \R \text{ is of class } C^1,
\end{align}
and there exist $r>0$ and $c_1>0$ such that for all $x\in \overline \Omega$, 
\begin{align} \label{H1} \tag{$H_1$}
\frac1p H_u(x,u,v)u + \frac1q H_v(x,u,v)v \ge  H(x,u,v)>0 \quad \text{ if } |(u,v)|\ge r,
\end{align}
\begin{align} \label{H2} \tag{$H_2$}
H(x,u,v)\le c_1(|u|^p+|v|^q) \quad \text{ if } |(u,v)|\le r, 
\end{align}
and for all $u,v\in\R$, 
\begin{align} \label{H3} \tag{$H_3$}
\begin{split}
|H_u(x,u,v)|&\le c_1(|u|^{p-1} + |v|^{(p-1)q/p}+1)\\
|H_v(x,u,v)|&\le c_1(|u|^{(q-1)p/q} + |v|^{q-1}+1).
\end{split}
\end{align}

Our main result establishes the existence of solutions to \eqref{P}.
\begin{thm}\label{MainThm}
Let $s\in (0,1)$ and let $p,q>1$ be such that \eqref{condpq.0} and \eqref{condpq.1} hold. Assume that $H$ satisfies \eqref{H0}-\eqref{H3}. Then there exists $\theta=\theta(p,q,n,s)\in (0,2)$ and a weak solution $(u,v)$ to \eqref{P} such that  $(u,v)\in H^{s\theta}_0(\Omega)\times H^{s(2-\theta)}_0(\Omega)$ to \eqref{P}.
\end{thm}

For a precise definition of weak solution to \eqref{P} see Section \ref{sols}.

Under more restricted assumptions on $p, q$ (which hold in particular when $p$ and $q$ are both subcritical),  the solution $(u,v)$ given by the previous Theorem is a finite energy solution and enjoy more regularity: 
\begin{cor}\label{Corollary_regularity}
Suppose that $(p,q)$ satisfy \eqref{condpq.0}, \eqref{condpq.1} and also 
$$ \begin{cases} 
\frac{p+q}{p(q-1)} \ge 2\frac{n-2s}{n+2s} \qquad \text{if } p\le q, \\
\frac{p+q}{q(p-1)} \ge 2\frac{n-2s}{n+2s} \qquad \text{if } p\ge q 
\end{cases}
$$
(which holds if $p,q\le 2n/(n-2s)$).  Then the solution given by Theorem \ref{MainThm}  belongs to $(H^{s\theta}_0(\Omega)\cap X)\times (H^{s(2-\theta)}_0(\Omega)\cap X)$ where $X=H^s_0(\Omega)\cap W^{2s,2n/(n+2s)}_{loc}(\Omega)$. 
\end{cor}

We show in Figure \ref{fig:plot_cor} the region in the $(p,q)$ plane where the corollary applies 
for the case when $n=5$ and $s=1/2$. 

\subsection*{Organization of the paper} The rest of the paper is organized as follows: In Section \ref{S2}, we collect all the preliminaries needed in the course of our arguments and review different notions of solutions to \eqref{P}. 

In Section \ref{S3} we state our main results and develop the functional framework needed to apply the variational methods. We end this section by providing a proof of the regularity result, Corollary \ref{Corollary_regularity}.

In Section \ref{S4}, we prove the main result of the paper, namely Theorem \ref{MainThm}.

We end this paper in Section \ref{S5} where we discuss some possible extensions and generalizations of our results.

\begin{figure}
\centering
\includegraphics[width=0.7\linewidth]{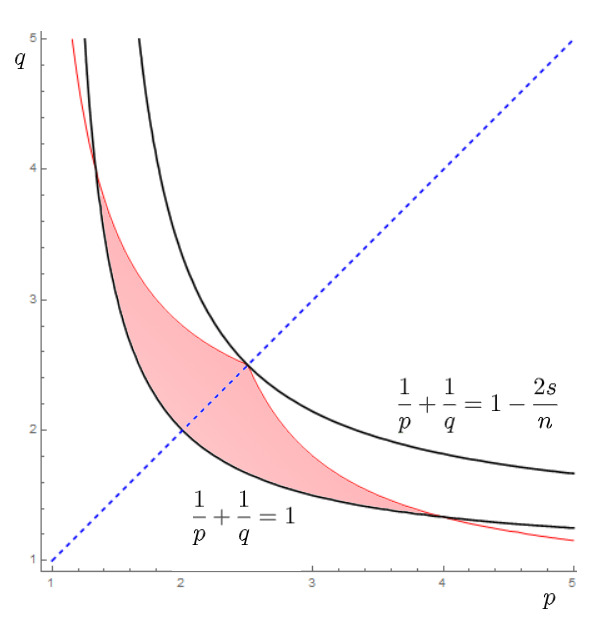}
\caption{Region of admissible $p$ and $q$ (in red) in Corollary \ref{Corollary_regularity} for $n=5$ and $s=1/2$.}
\label{fig:plot_cor}
\end{figure}

\section{Preliminaries}\label{S2}
In this section we collect all the preliminaries needed in the course of the work.

\subsection{Operator and spaces}
Given a smooth bounded open set $\Omega\subseteq \R^n$ we consider the fractional Sobolev spaces 
$$
H^s(\Omega)=\Big\{u\in L^2(\Omega)\colon [u]_s^2:=\iint_{\R^n\times \R^n}\frac{|u(x)-u(y)|^2}{|x-y|^{n+2s}}\,dxdy <\infty\Big\},
$$
endowed with norm $\|u\|_s:=\|u\|_2+[u]_s$, and denote $H^s_0(\Omega)$ the closure of $C_c^\infty(\Omega)$ for $\|.\|_s$. It turns out that $H^s_0(\Omega)=H^s(\Omega)$ when $s<\frac12$ and, in general when $s\in (0,1)$, that 
$$
H^s_0(\Omega)=\{u\in H^s(\R^n)\colon u=0 \text{ a.e. in } \R^n\setminus \Omega\}.
$$
We denote the dual space of $H^s_0(\Omega)$ by $H^{-s}(\Omega)$. We recall the Rellich-Kondrachov type result for fractional Sobolev spaces:
\begin{thm}\label{RK}
Let $2_t^* = 2n/(n-2t)$ if $n>2t$ and $2_t^*=\infty$ if $n\le 2t$.  Let $r\in [1,2_t^*)$, then $H^t_0(\Omega)\subset L^r(\Omega)$ with compact inclusion.
\end{thm}

For future use we will also need the higher order fractional Sobolev spaces $H_0^{1+\alpha}(\Omega)$, $\alpha\in (0,1)$, consisiting of functions $u\in H_0^1(\Omega)$ such that $\nabla u\in H_0^\alpha(\Omega)$, and endowed with the norm $\|u\|_{1+\alpha}=\|u\|_2 + \|\nabla u\|_\alpha$. 

The \emph{Gagliardo} fractional Laplacian $(-\Delta)^s u$, $s\in (0,1)$, of a function $u$ belonging to the Schwarz space $\mathcal{S}(\R^n)$ can be defined  via the Fourier transform $\F$ as 
\begin{align}\label{DefFracLap}
(-\Delta)^s u(x)&= \F^{-1}(|\xi|^{2s}\F(u))(x) =\text{p.v.} \int_{\R^n} \frac{u(x)-u(y)}{|x-y|^{n+2s}}\,dy
\end{align}
up to a normalization constant depending of $n$ and $s$.

In $\Omega$, we consider the fractional Laplacian $(-\Delta)^s$ with Dirichlet boundary condition in the sense that  $(-\Delta)^s$ denotes  the unbounded operator in $L^2(\Omega)$ with domain 
$$
\D((-\Delta)^s) = \{u\in H^s_0(\Omega)\colon |\xi|^{2s}\F(u)\in L^2(\R^n)\} = H^{2s}(\R^n)\cap H^s_0(\Omega).
$$
defined by \eqref{DefFracLap}. Then $(-\Delta)^s$ is a bounded operator between $H_0^s(\Omega)$ and $H^{-s}(\Omega)$, namely  
\begin{equation}\label{NotationWeak}
\langle (-\Delta)^s u,\varphi\rangle = \frac12\iint_{\R^n\times \R^n} \frac{(u(x)-u(y)) (\varphi(x)-\varphi(y))}{|x-y|^{n+2s}}\,dxdy
\end{equation} 
for every $\varphi \in H^s_0(\Omega)$. 

Notice that the r.h.s of \eqref{NotationWeak} makes sense when $u\in H_0^{\theta s}(\Omega)$ and $\varphi\in H_0^{(2-\theta)s}(\Omega)$  for any $\theta\in (0,2)$. Indeed, taking $u,\varphi\in C^\infty_c(\Omega)$ and omitting the factor $1/2$ for simplicity, we can  write it as 
\begin{eqnarray*} 
\iint_{\R^n\times \R^n} \frac{(\tau_hu(x)-u(x)) (\tau_h\varphi(x)-\varphi(x))}{|h|^{n+2s}}\,dxdh = \int_{\R^n} (\tau_hu-u,\tau_h\varphi-\varphi)_{L^2} \,\frac{dh}{|h|^{n+2s}}
\end{eqnarray*}
where $\tau_hu(x):=u(x-h)$. Passing to Fourier variable with $\F u(\xi)=\int_{\R^n} u(x)\exp(-2i\pi\xi x)\,dx$, so that $\F [\tau_h u](\xi) = \exp(-2i\pi\xi h)\F u(\xi)$, we obtain   
\begin{eqnarray*} 
\int_{\R^n} \Big(\int_{\R^n} \frac{|e^{-2i\pi\xi h}-1|^2 }{|h|^{n+2s}}\,dh\Big) \F u(\xi)\overline{\F \varphi(\xi)} \,d\xi.
\end{eqnarray*}
It is easy to see that the inner integral is equal to  $|\xi|^{2s}C(n,s)$ where $C(n,s):=\int_{\R^n} |e^{-2i\pi h_1}-1|^2 |h|^{-n-2s}\,dh$, see \cite{DPV}. Up to the constant $\frac12 C(n,s)$, the r.h.s of \eqref{NotationWeak} can thus be written as 
$$ 
\int_{\R^n} |\xi|^{2s}\F u(\xi)\overline{\F \varphi(\xi)} \,d\xi = \int_{\R^n} |\xi|^{\theta s}\F u(\xi)|\xi|^{(2-\theta) s}\overline{\F \varphi(\xi)} \,d\xi.
$$
If $\theta s<1$ and $(2-\theta) s<1$, i.e. $\theta\in (2-1/s, 1/s)\cap (0,2)$, then this is bounded by 
$$ 
\| |\xi|^{\theta s}\F u\|_2^2 \| |\xi|^{(2-\theta) s} \F \varphi\|_2^2  \lesssim [u]^2_{\theta s} [\varphi]^2_{(2-\theta) s} \le \|u\|_{\theta s} \|\varphi\|_{(2-\theta) s}.
$$ 
Otherwise, if $\theta\in (0,2)$ but, let's say $\theta s\in (1,2)$, so that $(2-\theta) s<1$, then writing $\theta s = 1+\alpha$ with $\alpha\in (0,1)$, we have 
\begin{eqnarray*}
\int_{\R^n} |\xi|^\alpha|\xi|\F u(\xi)|\xi|^{(2-\theta) s}\overline{\F \varphi(\xi)} \,d\xi 
& \lesssim & \||\xi|^\alpha \F[\nabla u]\|_2^2 \||\xi|^{(2-\theta) s}\F \varphi\|_2^2 \\
& \lesssim & [\nabla u]^2_{\alpha}[\varphi]^2_{(2-\theta) s} \\ 
& \le & \|u\|_{\theta s} \|\varphi\|_{(2-\theta) s}.
\end{eqnarray*} 
Thus for any $\theta\in (0,2)$, the r.h.s. of  \eqref{NotationWeak} is bounded (up to a constant depending only on $n$ and $s$) by $\|u\|_{\theta s} \|\varphi\|_{(2-\theta) s}$ and thus makes sense for $(u,\varphi)\in H_0^{\theta s}(\Omega)\times H_0^{(2-\theta) s}(\Omega)$. With a slight abuse of notation, we will keep on writing \eqref{NotationWeak} in this case.

Assuming now that $\varphi\in C^\infty_c(\Omega)$, so that  $\varphi\in H^{(2-\theta)s}(\Omega)$, and still assuming $u\in H^{\theta s}(\Omega)$, $\theta \in (0,2)$, notice also that 
\begin{equation}\label{WeakDistrib}
\langle (-\Delta)^s u,\varphi\rangle = (u,(-\Delta)^s\varphi)_{L^2}.
\end{equation}
Indeed as we just saw, these regularity assumptions ensure that $\frac{(u(x)-u(y)) (\varphi(x)-\varphi(y))}{|x-y|^{n+2s}}$ is integrable in $\R^n\times\R^n$, and so is  $u(x) \frac{\varphi(x)-\varphi(y)}{|x-y|^{n+2s}}$ in particular. Thus 
\begin{align*} 
\int_{\R^n} u (-\Delta)^s\phi \,dx 
&= \int_{\R^n} u(x)  \lim_{\varepsilon\to 0} \int_{\{|x-y|>\varepsilon\}} \frac{\varphi(x)-\varphi(y)}{|x-y|^{n+2s}} \,dydx  \\
&= \lim_{\varepsilon\to 0} \iint_{\{|x-y|>\varepsilon\}}  u(x)   \frac{\varphi(x)-\varphi(y)}{|x-y|^{n+2s}} \,dydx  \\ 
&= \lim_{\varepsilon\to 0} \frac12\iint_{\{|x-y|>\varepsilon\}}  \frac{(u(x)-u(y)) (\varphi(x)-\varphi(y))}{|x-y|^{n+2s}} \,dydx \\
& = \langle (-\Delta)^s u,\varphi\rangle. 
\end{align*}
These remarks are summarized in the following proposition: 

\begin{prop} \label{PropWeakDistrib}
Assume that $u\in H^{\theta s}(\Omega)$  for some $\theta \in (0,2)$ and $\varphi\in C^\infty_c(\Omega)$. Then 
$$ 
\int_{\R^n} u(-\Delta)^s\varphi \,dx = \langle (-\Delta)^s u,\varphi\rangle, 
$$
where $\langle (-\Delta)^s u,\varphi\rangle$ is defined in \eqref{NotationWeak}. 
\end{prop}

\subsection{Notions of solutions}\label{sols}
Given a smooth, bounded and open set $\Omega\subset \R^n$, $s\in(0,1)$ and $f\in L^1_{loc}(\Omega)$,  we consider the Dirichlet problem
\begin{align}\label{eq.f}
\begin{cases}
(-\Delta)^s u=f &\text{ in }\Omega\\
u=0 &\text{ in } \R^n\setminus \Omega
\end{cases}
\end{align}
Several notions of solutions exist depending on the regularity of $f$ and $u$. 

First, if $f\in H^{-s}(\Omega)$ then, recalling that $(-\Delta)^s:H^s_0(\Omega)\to H^{-s}(\Omega)$ is a bounded operator, we can look for a solution $u\in H^s_0(\Omega)$:  
\begin{defn}[Definition 11 in \cite{Leonori}] \label{finite_energy}
Let $f\in H^{-s}(\Omega)$. A function $u\in H^s_0(\Omega)$ is a \emph{finite energy solution} to \eqref{eq.f} if for every $\phi\in H^s_0(\Omega)$ it holds that
$$ 
\langle (-\Delta)^s u,\varphi \rangle  =\langle f, \phi\rangle_{H^{-s}(\Omega),H^s_0(\Omega)}.
$$
\end{defn}

Since  $\langle (-\Delta)^s u,\varphi \rangle$ defines a scalar product equivalent to the standard one in $H^s_0(\Omega)$,  Riesz' Theorem ensures the existence of a unique energy solution to \eqref{eq.f} for every $f\in H^{-s}(\Omega)$. Notice that $L^p(\Omega)\hookrightarrow H^{-s}(\Omega)$ for any $p\ge 2n/(n+2s)$.  

When $f$ is only locally integrable we consider weak and distributional solutions.
\begin{defn}
Let $f\in L^1_{loc}(\Omega)$. A function $u$ belonging to $H^{\theta s}_0(\Omega)$ for some $\theta \in (0,2)$ is a \emph{weak solution} to \eqref{eq.f} if
$$ 
\langle (-\Delta)^s u,\varphi \rangle =\int_\Omega f \varphi\,dx \quad \text{ for all } \varphi \in C_c^\infty(\R^n), 
$$
where $\langle (-\Delta)^s u,\varphi \rangle$ is defined in \eqref{NotationWeak}.
\end{defn}

A finite energy solution is a weak solution. However a weak solution $u$ is in general not an energy solution except of course if it belongs to $H^s_0(\Omega)$ i.e. $\theta \ge 1$. 

Notice that $(-\Delta)^s\phi$ is a bounded function if $\phi\in C_c^\infty(\Omega)$. As a consequence $\int_\Omega u(-\Delta)^s\phi\,dx$ exists if $u\in L^1(\Omega)$. This leads to the notion of distributional solution: 
\begin{defn}
Let $f\in L^1_{loc}(\Omega)$. A function $u\in L^1(\R^n)$ is a \emph{distributional solution} to \eqref{eq.f} if $u=0$ a.e. in $\R^n\setminus \Omega$ and for every $\phi\in C_c^\infty(\Omega)$ it holds that
$$
\int_{\R^n} u (-\Delta)^s \phi\,dx=\int_\Omega f\phi\,dx.
$$
\end{defn}

Notice that a weak solution to \eqref{eq.f}  is a distributional solution to \eqref{eq.f} by Prop, \ref{PropWeakDistrib}. The converse is false in general unless $u\in H^{\theta s}_0(\Omega)$:  
\begin{prop}\label{propDistrib2Weak} 
A distributional solution belonging to $H^{\theta s}_0(\Omega)$ for some $\theta \in (0,2)$ is a weak solution. 
\end{prop}

Given a smooth, bounded and open set  $\Omega\subset \R^n$ and $s\in(0,1)$, consider the system \eqref{P}
\begin{defn} \label{def.weak.sistem}
Let $H\colon \R^2\times \overline \Omega \to \R$ be a $C^1$ function. We say that $(u,v)\in L^1_{loc}(\R^n)\times L^1_{loc}(\R^n)$ is a \emph{weak} (resp. \emph{distributional}) solution to \eqref{P} if  $H_u(., u, v)$ and $H_v(., u, v)$ belong to $L^1_{loc}(\R^n)$ and the functions $u$ and $v$ are weak (resp. distributional) solutions of each corresponding equation. 
\end{defn}

More precisely,  $(u,v)$ is a weak solution  to \eqref{P} if $u\in H^{\theta s}_0(\Omega)$ and $v\in H^{\theta' s}_0(\Omega)$ for some $\theta,\theta' \in (0,2)$, $H_u(.,u,v)$ and  $H_v(.,u,v)$ are in $L^1_{loc}(\Omega)$, and it holds that
$$
\langle (-\Delta)^s u,\varphi \rangle = \int_\Omega H_v(x,u,v) \varphi(x)\,dx, \qquad \langle (-\Delta)^s v,\psi \rangle = \int_\Omega H_u(x,u,v) \psi(x)\,dx
$$
for every $\varphi,\psi \in C_c^\infty(\R^n)$. It follows from Proposition \ref{propDistrib2Weak} that 
\begin{prop}\label{propDistrib2Weak2} 
A distributional solution $(u,v)$  belonging to $H^{\theta s}_0(\Omega)\times H^{\theta' s}_0(\Omega)$ for some $\theta,\theta' \in (0,2)$ is a weak solution. 
\end{prop}

\subsection{Interpolation spaces}\label{InterSpaces}
Denote $\{\lambda_{k,s}\}_{k\in\N}$ the eigenvalues of $(-\Delta)^s$ and $\{\phi_{k,s}\}_{k\in\N}$ associated eigenfunctions forming an orthonormal basis of $L^2(\Omega)$. We can then consider the fractional power $A_s^\theta \colon \D(A_s^\theta)\subset L^2(\Omega)\to L^2(\Omega)$, $\theta\in [0,2]$, defined as 
\begin{equation}\label{DefA}
A_s^\theta u = \sum_{k=1}^\infty \lambda_{k,s}^{\theta/2} u_k \phi_{k,s}, \qquad u_k =(u \varphi_{k,s})_{L^2}
\end{equation}
with domain 
\begin{equation}\label{DefE}
E_s^\theta := \D(A_s^\theta) =:\{u\in L^2(\Omega)\colon \sum_{k=1}^\infty \lambda_{k,s}^\theta |u_k|^2 <\infty\}.
 \end{equation}
Then $E_s^\theta$ is a Hilbert space with inner product and associated norm given by
$$ 
(u,v)_{E^\theta_s}=(A_s^\theta u, A_s^\theta v)_{L^2} = \sum_k \lambda_{k,s}^\theta u_kv_k, \qquad \|u\|_{E_s^\theta} = \int_\Omega |A_s^\theta u|^s\,dx =  \sum_{k=1}^\infty \lambda_{k,s}^\theta |u_k|^2.  
$$
Moreover  they are interpolation spaces in the sense of Lions-Magenes,
$$ 
E_s^\theta = [H^s_0(\Omega), L^2(\Omega)]_{1-\theta} 
$$
and $E_s^\theta \subset H_0^{s\theta}(\Omega)$ (with equality if $s\neq \frac12$ and $s\theta\neq \frac12, \frac32$). Notice in particular that $E_s^\theta \subset L^r(\Omega)$ if $r\le 2^*_{s\theta}$ with compact inclusion if $r<2^*_{s\theta}$.

Notice eventually that $A_s^\theta u :  E_s^\theta\to L^2(\Omega$ is an isomorphism. Denoting $A_s^{-\theta}$ its inverse, we have 
$$ 
(A^{-\theta}_s w,y)_{E_s^\theta} = (w. A^\theta y)_{L^2} \qquad w\in L^2,\,y\in E_s^\theta. 
$$

\section{Variational framework}\label{S3}
In this section we introduce the variational framework needed in order to apply the variational techniques.

First let us fix some $s\in (0,1)$ and real numbers $p,q>1$ satisfying \eqref{condpq.0}, \eqref{condpq.1}. Under these assumptions we can then find $\theta\in (0,2)$ such that
\begin{align} \label{condpq} 
p<2_{s\theta}^* := \frac{2n}{n-2s\theta}, \qquad q<2_{s(2-\theta)}^* := \frac{2n}{n-2s(2-\theta)}.
\end{align}
We then have the following continuous embeddings
\begin{equation}\label{ContEmbed}
\begin{split}
& E_s^\theta \hookrightarrow  H^{s\theta}_0(\Omega) \hookrightarrow L^{2^*_{s\theta}}(\Omega) \hookrightarrow L^p(\Omega), \\
& E_s^{(2-\theta)} \hookrightarrow  H^{s(2-\theta)}_0(\Omega) \hookrightarrow L^{2^*_{s(2-\theta)}}(\Omega) 
\hookrightarrow L^q(\Omega), 
\end{split}
\end{equation} 
the last embedding of each chain being compact. 

We will find a solution of \eqref{P} by looking for a critical point of the energy functional $J_\theta\colon E_s^\theta\times E_s^{2-\theta}\to \R$ defined as
\begin{equation}\label{J}
J_\theta (u,v) = \int_\Omega A_s^\theta u A_s^{2-\theta} v\, dx - \int_\Omega H(x,u,v)\, dx,
\end{equation}
where $A_s^\theta$ and $ E_s^\theta$ are defined in \eqref{DefA} and \eqref{DefE} respectively (we will verify in the next section that the growth condition \eqref{H3} and \eqref{ContEmbed} guarantee that $J_\theta$ is a well-defined $C^1$ functional on $E_s^\theta\times E_s^{2-\theta}$). 

Once this value of $\theta$ is defined, we have the notion of solution associated to it. Later on we will relate this concept of solution to the former ones previously introduced.
\begin{defn}\label{theta.sol}
A critical point  $(u,v)\in E_s^\theta\times E_s^{2-\theta}$ of $J_\theta$ will be called a \emph{$\theta-$weak solution} to \eqref{P}.  It satisfies 
\begin{equation} \label{J'}
\int_\Omega \left( A_s^\theta u A_s^{2-\theta} \psi + A_s^\theta\varphi A_s^{2-\theta} v \right)\,dx - \int_\Omega \left( H_u(x,u,v)\varphi + H_v(x,u,v)\psi\right)\,dx =0
\end{equation}
for any test functions $(\varphi,\psi)\in E_s^\theta\times E^{2-\theta}_s$.
\end{defn}

We will prove in the next section that 
\begin{thm} \label{teo1}
There exists a $\theta-$weak solution to \eqref{P}.
\end{thm}

Theorem \ref{MainThm} then follows by noticing that a $\theta-$weak solution is in fact a weak solution: 
\begin{thm} \label{teo2}
If $(u,v)\in E_s^\theta\times E_s^{2-\theta}$ is a $\theta-$weak solution to \eqref{P}, then $(u,v)$ is a weak solution to \eqref{P}.
\end{thm}

\begin{proof}
Let $(u,v)$ be a $\theta-$weak solution to \eqref{P} i.e. $(u,v)\in E_s^\theta \times E_s^{2-\theta}$ satisfies \eqref{J'}.  First, by hypothesis \eqref{H3} we have that $H_u(\cdot, u, v), H_v(\cdot, u, v)\in L^1(\Omega)$. Indeed, Young's inequality and \eqref{H3} gives $|H_u(x,u,v)|, |H_v(x,u,v)|\le C(|u|^p+|v|^q+1)$ and, with \eqref{ContEmbed}, 
$$ 
\int_\Omega |u|^p+|v|^q+1\,dx \le  C(\|u\|_{E_s^\theta}^p+\|v\|_{E_s^{2-\theta}}^q+1).
$$

To conclude, in view of Prop. \ref{propDistrib2Weak2}, it suffices to show  that $(u,v)$ is a distributional solution. Taking $\psi=0$ in \eqref{J'} gives that
\begin{equation}\label{Eq100}
\int_\Omega A_s^\theta\varphi A_s^{2-\theta} v\,dx = \int_\Omega  H_u(x,u,v)\varphi \,dx \qquad\text{for any $\varphi\in E_s^\theta$. }
\end{equation}
Since $\{\phi_{k,s}\}$ is an orthonormal basis of $L^2(\Omega)$ we can rewrrite the l.h.s. of \eqref{Eq100} as 
$$ 
\left(\sum_{k\ge 1} \lambda_{k,s}^\frac{\theta}{2} \varphi_k \phi_{k,s},  \sum_{j\ge 1} \lambda_{j,s}^\frac{2-\theta}{2} v_j \phi_{j,s} \right)_{L^2} = \sum_{k\ge 1} \lambda_{k,s} \varphi_k v_k. 
$$
where we denoted $ \varphi_k=( \varphi,\phi_{k,s})_{L^2}$ and $v_j=( v_j,\phi_{j,s})_{L^2}$. Taking $\varphi\in C^\infty_c(\Omega)$, so that $(-\Delta)^s \varphi = \sum_{k\ge 1} \lambda_{k,s} \varphi_k\in L^2(\Omega)$, we can rewrite \eqref{Eq100} as  
$$ 
\int_\Omega v(-\Delta)^s \varphi\,dx= \int_\Omega  H_u(x,u,v)\varphi \,dx .
$$
This shows that $v$ is a distributional solution. We can verify in the same way that $u$ is a distributional solution  taking $\varphi=0$ in \eqref{J'}.
\end{proof}

We finish this section with the proof of the regularity result for weak solutions to \eqref{P}, that is Corollary \ref{Corollary_regularity}. 

\begin{proof}[Proof of Corollary \ref{Corollary_regularity}]
We first recall that if $f\in L^p(\Omega)$ for some $p\ge 2n/(n+2s)$, then $f\in H^{-s}(\Omega)$ so that the equation $(-\Delta)^s u=f$ in $\Omega$ has a unique finite energy solution $u\in H^s_0(\Omega)$ (see Definition \ref{finite_energy}). According to \cite[Theorem 1.4]{BWZ17}, $u$ belongs in fact to $u\in W^{2s,p}_{loc}(\Omega)$. 

The proof of Corollary \ref{Corollary_regularity} will thus follow once we can prove that $H_u(\cdot, u, v)$ and $H_v(\cdot, u, v)$ belong to $L^{2n/(n+2s)}(\Omega)$, where $(u,v)$ is the solution given by Theorem \ref{MainThm}. Let us assume w.l.o.g. that $p\le q$. In view of the growth assumption \eqref{H3}, we thus need $u\in L^{p\frac{q-1}{q}\frac{2n}{n+2s}}(\Omega$ and $v\in L^{(q-1)\frac{2n}{n+2s}}(\Omega)$. Since $u\in H^{s\theta}\Omega)\hookrightarrow L^\frac{2n}{n-2s\theta}(\Omega)$ and $v\in H^{s\theta}\Omega)\hookrightarrow L^\frac{2n}{n-2s(2-\theta)}(\Omega)$, we  need 
\begin{equation}\label{ineq}
 p\frac{q-1}{q}\frac{2n}{n+2s} \le \frac{2n}{n-2s\theta},\qquad \text{and} \qquad (q-1)\frac{2n}{n+2s} \le \frac{2n}{n-2s(2-\theta)}. 
\end{equation}

Notice that if $q\le \frac{2n}{n-2s}$ then $\frac{q-1}{q}\frac{2n}{n+2s}\le 1$ so that \eqref{ineq} is a consequences of \eqref{condpq}. 

When  $q>\frac{2n}{n-2s}$ then \eqref{ineq} implies \eqref{condpq}. Notice that \eqref{ineq} holds for some $\theta\in (0,2)$ if we assume that 
$$ 
\frac{1}{q-1} + \frac{q}{p(q-1)} \ge 2\frac{n-2s}{n+2s}. 
$$
Notice eventually that, since  $q\ge p$, this inequality holds in particular if $q<2n/(n-2s)$. 
\end{proof}

\section{Proof of Theorem \ref{teo1}: Existence of $\theta$-weak solution.}\label{S4}
In order to finish the proof of our main result, Theorem \ref{MainThm}, it remains to prove the existence of a $\theta-$weak solution i.e. the  energy function
\begin{equation}\label{J2}
J_\theta (u,v) = \int_\Omega A_s^\theta u A_s^{2-\theta} v\, dx - \int_\Omega H(x,u,v)\, dx \qquad (u,v)\in  E_s^\theta\times E_s^{2-\theta}
\end{equation} 
possesses a critical point. Recall that $\theta\in (0,2)$ satisfies \eqref{condpq}. We start by proving standard properties of  $J_\theta$.

\begin{prop}\label{JC1}
The functional $J_\theta$ defined in \eqref{J} is of class $C^1$ and its derivative is given by 
$$
J'_\theta(u,v)(\varphi,\psi)=\int_\Omega \left( A_s^\theta u A_s^{2-\theta} \psi + A_s^\theta\varphi A_s^{2-\theta} v - H_u(x,u,v)\varphi - H_v(x,u,v)\psi\right)\,dx.
$$
for every $(\varphi,\psi)\in E^\theta_s\times E_s^{2-\theta}$.
\end{prop}

\begin{proof}
Let $(u,v), (\varphi,\psi)\in E^\theta_s \times E^{2-\theta}_s$ and define the corresponding Fourier coefficients
$$
u_k=(u,\phi_{k,s}), \quad v_k=(v,\phi_{k,s}), \quad \psi_k=(\psi,\phi_{k,s}), \quad \varphi_k=(\varphi,\phi_{k,s}).
$$
Then, 
\begin{align*}
&J_\theta(u+t\varphi,v+t\psi) \\ 
&= \int_\Omega \left( \sum_{k=1}^\infty  \lambda_{k,s}^\frac{\theta}{2}(u_k+t\varphi_k)\phi_{k,s} 
\sum_{j=1}^\infty \lambda_{j,s}^\frac{2-\theta}{2}(v_j+t\psi_j)\phi_{j,s} -H(x,u+t\varphi,v+t\psi)\right)\,dx.
\end{align*}
Taking the derivative at $t=0$ gives 
\begin{align*}
&\frac{d}{dt}\Big|_{t=0} J_\theta(u+t\varphi,v+t\psi) \\
&=
\int_\Omega \left(  \sum_{k=1}^\infty \lambda_{k,s}^\frac{\theta}{2}\varphi_k \phi_{k,s} \sum_{j=1}^\infty \lambda_{j,s}^\frac{2-\theta}{2}b_j\phi_{j,s}+ \sum_{k=1}^\infty \lambda_{k,s}^\frac{\theta}{2}a_k \phi_{k,s} \sum_{j=1}^\infty \lambda_{j,s}^\frac{2-\theta}{2}\psi_j\phi_{j,s} \right)\,dx\\
&\qquad - \int_\Omega \left( H_u(x,u,v)\varphi + H_v(x,u,v)\psi \right)\,dx\\
&=\int_\Omega \left( A_s^\theta u A_s^{2-\theta} \psi + A_s^\theta\varphi A_s^{2-\theta} v - H_u(x,u,v)\varphi - H_v(x,u,v)\psi\right)\,dx.
\end{align*}
The terms involving the derivatives of $H$ are well defined. Indeed by \eqref{ContEmbed} and H\"older's inequality, 
\begin{align*}
\int_\Omega |u|^{p-1}|\varphi|\,dx 
&\leq \|u\|_{L^p(\Omega)}^{p-1} \|\varphi\|_{L^p(\Omega)} 
\leq	\|u\|_{E_s^\theta}^{p-1} \|\varphi\|_{E_s^\theta}, \\
\int_\Omega |v|^{q\frac{p-1}{p}}|\varphi|\,dx 
&\leq \|v\|_{L^q(\Omega)}^{\frac{(p-1)q}{p}} \|\varphi\|_{L^p(\Omega)}
 \leq	 \|v\|_{E_s^{2-\theta}}^{\frac{(p-1)q}{p}} \|\varphi\|_{E_s^{\theta}}.
\end{align*}
Therefore, using assumption \eqref{H3}   we have that
\begin{align} \label{hu1}
\begin{split}
\int_\Omega |H_u(x,u,v)|\varphi|\,dx 
&\leq C ( \|u\|_{E_s^\theta}^{p-1}  + \|v\|_{E_s^{2-\theta}}^{\frac{(p-1)q}{p}} +1 )  \|\varphi\|_{E_s^\theta}.
\end{split}
\end{align}
Similarly, 
\begin{align} \label{hv1}
\begin{split}
\int_\Omega |H_v(x,u,v)\psi|\,dx 
&\leq C ( \|v\|_{E_s^{2-\theta}}^{q-1} + \|u\|_{E_s^{\theta}(\Omega)}^{\frac{(q-1)p}{q}} +1 )  \|\psi\|_{E_s^{2-\theta}}.
\end{split}
\end{align}
Usual arguments show that $J_\theta$ is Fr\'echet differentiable and $J_\theta'$ is continuous. 
\end{proof}

Following sections 2 and 3 of \cite{DF-F}, we will prove that $J_\theta$ has a critical point applying the 
following minmax Theorem (proved in \cite{F}):

\begin{thm}\label{minimax}
Let $E$ be a Hilbert space such that $E = X \oplus Y$ for some subspaces $X,Y\subset E$. Let $\Phi: E \to \R$ be a $C^1$ functional satisfying the Palais-Smale condition and having the following structure:
$$ 
\Phi(z)=\frac{1}{2}\langle Lz,z \rangle +\mathcal{H}(z), 
$$
where 
\begin{equation}\label{I1} \tag{$I_1$}
L\colon E\to E\text{ is a linear, bounded and self adjoint operator,} 
\end{equation}
\begin{equation}\label{I2} \tag{$I_2$}
\mathcal{H}'\text{ is compact,} 
\end{equation}
and there exist two linear, bounded, invertible operators $B_1,B_2:E\to E$ such that for any $\omega\ge 0$ the operator 
\begin{equation}\label{I3} \tag{$I_3$}
\widehat{B}(\omega)=P_XB^{-1}_1exp(\omega L)B_2\colon X\to X 
\end{equation}
is invertible. Here $P_x$ denotes the projection of $E$ onto $X$.

Furthermore, given $\rho>0$, $z_+\in Y\backslash\{0\}$, $\sigma>\rho/\|B_1^{-1}B_2 z_+\|$ and $M>\rho$, define 
\begin{align*} 
& S=\{B_1z\colon \|z\|=\rho,\, z\in Y \}, \\
& Q=\{B_2(tz_+ + z)\colon 0\leq t\leq\sigma,\, \|z\|\leq M, \, z\in X\}. 
\end{align*} 
Assume eventually there exists $\delta>0$ such that 
\begin{equation}\label{I4}\tag{$I_4$}
\Phi(z)\geq\delta \qquad \forall z\in S,
\end{equation}
\begin{equation}\label{I5}\tag{$I_6$}
\Phi(z)\leq 0 \qquad  \forall z\in\partial Q
\end{equation}
where $\partial Q$ denotes the boundary of $Q$ relative to  $\{B_2(tz_+ + z)\colon t\in \R,\,z\in X\}$.

Then $\Phi$ possesses a critical point with critical value greater or equal than $\delta$.
\end{thm}

In order to be self-contained, we will sketch the proof of all the hypotheses, only providing the proof of those that we consider important. For a more detailed proof, please refer to sections 1, 2, and 3 of \cite{DF-F}.

Consider the Hilbert space $E:=E_s^\theta\times E_s^{2-\theta}$ with the inner product $(.,.)_E$ induced by those of $E^\theta$ and $E^{2-\theta}$. The continuous and symmetric bilinear form $B\colon E\times E\to \R$ defined as
$$ 
B((u,v),(\varphi,\psi))=\int_\Omega \left( A_s^\theta u A_s^{2-\theta} \psi + A_s^\theta\varphi A_s^{2-\theta} v \right)\,dx 
$$
induces the self adjoint bounded linear operator $L\colon E\to E$ defined by 
$$ 
(Lz,\nu)_E=B(z,\nu), 
$$
which is explicitly given by 
\begin{equation}\label{operatorL}
L(u, v) = ((A_s^\theta)^{-1}A_s^{2-\theta}v,(A_s^{2-\theta})^{-1}A_s^\theta u). 
\end{equation}
Introducing  the quadratic form 
$$  
Q(z)=\frac{1}{2}(Lz,z)_{E_s}=\int_\Omega A_s^{\theta}uA_s^{2-\theta}v\, dx \quad \forall\ z=(u,v)\in E, 
$$
and the Nemytskii operator 
$$ 
\mathcal{H}(z) = \int_\Omega H(x,u,v)\,dx \qquad z=(u,v), 
$$
we have 
$$ 
J_\theta(z) = Q(z) - \mathcal{H}(z). 
$$ 
Next, from hypothesis \ref{H0} and \eqref{H3}, and \eqref{ContEmbed} Proposition \ref{JC1}, $\mathcal{H}$ is $C^1$ with $\mathcal{H'}$ compact. Assumptions \eqref{I1} and \eqref{I2} are thus satisfied. 

Now, we focus on the eigenvalue problem for $L$. In view of \eqref{operatorL}, 
\begin{equation}\label{Lz}
L(u,v)=\lambda (u,v)  \Leftrightarrow 
\begin{cases}
(A_s^\theta)^{-1}A_s^{2-\theta}v=\lambda u\\
(A_s^{2-\theta})^{-1}A_s^\theta u=\lambda v.
\end{cases}
\end{equation}
Since $A_s^\theta$ and $A_s^{2-\theta}$ are isomorphisms of $L^2(\Omega)$, $\lambda$ cannot be zero and we get the following equality
\begin{equation}
v=\lambda^2 v.
\end{equation}
Therefore, $\lambda=1$ or $\lambda=-1$ with the corresponding eigenspaces 
$$ 
E^+=\{ (u,(A_s^{2-\theta})^{-1}A_s^\theta u)\colon u\in E_s^{\theta}\} 
$$
and
$$ 
E^-=\{ (u,-(A_s^{2-\theta})^{-1}A_s^\theta u)\colon u\in E_s^{\theta}\}. 
$$
Notice that $E^+$ and $E^-$ are orthogonal for $(\cdot, \cdot)_E$ and $E_s=E_s^+ \oplus E_s^-$. Moreover for $z=z^++z^-$, $z^{\pm}\in E_s^{\pm}$, 
\begin{equation}
B(z^+,z^-)=0, \qquad \frac12 \|z\|_{E_s}^2=Q(z^+)-Q(z^-). 
\end{equation}

Next, we prove the Palais-Smale condition for $J_\theta$.
\begin{prop}
$J_\theta$ satisfies the Palais-Smale condition.
\end{prop}

\begin{proof}
Let $\{(u_n,v_n)\}_{n\in\N}$ be a sequence in $E_s^{\theta}\times E_s^{2-\theta}$ such that 
$$ 
|J_\theta(u_n,v_n)|\leq C\text{ and }J'_\theta(u_n,v_n)\to0\text{ as }n\to\infty. 
$$ 
The condition on the derivative implies that there is a sequence $\{\varepsilon_n\}$ converging to 0 such that
\begin{equation}\label{PSDer} 
|(J'_\theta(u_n,v_n),\eta)|\leq \varepsilon_n \|\eta\|_{E_s^{\theta}\times E_s^{2-\theta}}
\quad \forall\eta\in E_s^{\theta}\times E_s^{2-\theta}.
\end{equation}
Take 
$$ 
\eta_n=\frac{pq}{p+q}\left(\frac{u_n}{p},\frac{v_n}{q}\right),  
$$
we get 
\begin{align*}
c+\varepsilon_n\|\eta_n\|
& \ge J_\theta(u_n,v_n) - J'_\theta(u_n,v_n)\eta_n  \\ 
& \geq\frac{pq}{p+q}\int_\Omega\frac{1}{p} H_u(u_n,v_n,x)u_n 
+\frac{1}{q}H_v(u_n,v_n,x)v_n-H(u_n,v_n,x)\,dx\\
&\hspace{0.5cm} + \left(\frac{pq}{p+q}-1\right)\int_\Omega H(u_n,v_n,x)\,dx.
\end{align*}
In view of  \eqref{H1} and \eqref{condpq.0} we deduce that 
$$ 
C(1+\|(u_n,v_n)\|)\geq\int_\Omega H(u_n,v_n,x)\,dx. 
$$ 
Independently, it follows from \eqref{H1} that $H(x,u,v)\ge C(|u|^p+|v|^q)-C$. Thus  
\begin{equation}\label{Eq200} 
\int_\Omega |u_n|^p+|v_n|^q\,dx \leq C(1+\|u_n\|_{E_s^\theta}+\|v_n\|_{E_s^{2-\theta}} ). 
\end{equation} 

Next, taking $\eta=(\phi,0)$ with $\phi\in E_s^\theta$ in \eqref{PSDer} gives 
$$ 
\left|\int_\Omega A_s^\theta \phi A_s^{2-\theta} v_n\, dx\right| \leq \int_\Omega\left|H_u(u_n,v_n,x)\phi\right|\, dx+\varepsilon_n\|\phi\|_{E_s^\theta}. 
$$
We can bound the integral in the r.h.s. using \eqref{hu1} to get 
$$
\left|\int_\Omega A_s^\theta \phi A_s^{2-\theta} v_n\, dx\right| \leq c_4(\|u_n\|^{p-1}_{L^p(\Omega)}+\|v_n\|^{\frac{(p-1)q}{p}}_{L^q(\Omega)}+1)\|\phi\|_{E_s^\theta}	\qquad \phi\in E_s^\theta.
$$
Since $A_s^\theta$ is an isometry between $E_s^\theta$ and $L^2$, we can rewrite this inequalty as 
$$
(A_s^{2-\theta} v_n,\psi)_{L^2} \leq c_4(\|u_n\|^{p-1}_{L^p(\Omega)}+\|v_n\|^{\frac{(p-1)q}{p}}_{L^q(\Omega)}+1)\|\psi\|_{L^2}	 \qquad \psi\in L^2
$$
from which it follows that 
\begin{equation}\label{Eq300}
\|v_n\|_{E_s^{(2-\theta)}} = \|A_s^{2-\theta} v_n\|_{L^2}  \leq c_4(\|u_n\|^{p-1}_{L^p(\Omega)}+\|v_n\|^\frac{(p-1)q}{p}_{L^q(\Omega)}+1).
\end{equation} 
Analogously,
\begin{equation}\label{Eq301}
\|u_n\|_{E_s^{\theta}} \leq c_5(\|v_n\|^{q-1}_{L^q(\Omega)}+\|u_n\|^\frac{p(q-1)}{q}_{L^q(\Omega)}+1).
\end{equation} 
Plugging these two estimates in \eqref{Eq200} gives 
\begin{align*}
\|u_n\|_p^p + \|v_n\|_q^q  \le C(1+\|u_n\|^{p-1}_{L^p(\Omega)}+\|v_n\|^{\frac{q(p-1)}{p}}_{L^q(\Omega)} +\|v_n\|^{q-1}_{L^q(\Omega)}+\|u_n\|^{\frac{p(q-1)}{q}}_{L^q(\Omega)}). 
\end{align*} 
It follows that the sequence $\{(u_n,v_n) \}$ is bounded in $L^p(\Omega)\times L^p(\Omega)$. Coming back to \eqref{Eq300}-\eqref{Eq301}, we finally deduce that $\{(u_n,v_n) \}$ is bounded in $E_s^\theta\times E_s^{2-\theta}$.

Finally, from the compactness of $\mathcal{H'}$ and  the invertibility of $L$ there exists a subsequence that converge in $E_s^\theta\times E_s^{2-\theta}$.	
\end{proof}

Now, we consider the study of geometric characteristics of $J_\theta$ . 

The orthogonal projections $P^{\pm}\colon E_s^\theta\times E_s^{2-\theta}\to E^{\pm}$ are defined as
\begin{equation*}
P^{\pm}(u,v)=\frac{1}{2}(u\pm A^{-\theta}A^{2-\theta}v,v\pm A^{-(2-\theta)}A^{\theta}u).
\end{equation*}
where we denoted $A^{-\theta}=(A^{\theta})^{-1}$ and $A^{-(2-\theta)}=(A^{2-\theta})^{-1}$.

Now we choose numbers $\mu>1,\nu>1$ such that
\begin{equation}\label{munu}
\frac{1}{p}<\frac{\mu}{\mu+\nu}\text{ and }\frac{1}{q}<\frac{\nu}{\mu+\nu}.
\end{equation}
and define $B_1(u,v)=(\rho^{\mu -1}u,\rho^{\nu-1}v)$, $B_2(u,v)=\left(\sigma^{\mu-1}u, \sigma^{\nu-1}v\right)$. Eventually we take some $(u_+, v_+)\in E^{+}$ with $u_+$ an eigenvector of $-\Delta^s$. Then \eqref{I4}-\eqref{I5} hold. Finally, we can conclude that $J_\theta$ satisfies all the hypotheses of Theorem \ref{minimax}. Therefore, it has a critical point.

\section{Extensions.}\label{S5}
The proof of \cite{DF-F} in which our existence result Theorem \ref{MainThm} is based, can be adapted to a wider class of operators. Indeed we can replace $(-\Delta)^s$ by a symmetric operator $\mathcal{L}\colon Dom(\mathcal{L}) \subset L^2(\Omega)\to L^2(\Omega)$ that can be diagonalized with positive eigenvalues in an orthonormal basis $\{\phi_k\}$ of $L^2(\Omega)$. Assume also $Dom(\mathcal{L})$ contains $C^\infty_c(\Omega)$ and is continuously embedded into $H^{2s}(\Omega)$ for some $s>0$. We can then consider as in section \ref{InterSpaces} the fractional power $A_s^\theta \colon \D(A_s^\theta)\subset L^2(\Omega)\to L^2(\Omega)$, $\theta\in [0,2]$, of $\mathcal{L}^\frac12$. Choosing $p,q$ satisfying \eqref{condpq}, the same proof as before gives the existence of a distributional solution to the system 
\begin{align*} 
& \mathcal{L}u = H_u(x,u,v), \\ 
& \mathcal{L}v = H_v(x,u,v) 
\end{align*} 
(the boundary condition is incorporated in the definition of $Dom(\mathcal{L})$). Here distributional solution means that 
\begin{align*} 
& (u,\mathcal{L}\phi)_{L^2} = (H_u(x,u,v),\phi)_{L^2}, \\ 
& (v,\mathcal{L}\phi)_{L^2} = (H_v(x,u,v),\phi)_{L^2}
\end{align*} 
for any $\phi\in C^\infty_c(\Omega)$. 

As examples of such $\mathcal{L}$ we can consider the following immediate generalization of $(-\Delta)^s$: 
$$ 
\mathcal{L}_\mathcal{K}u(x) = PV \int_{\R^n} (u(x)-u(y))\mathcal{K}(x,y)\,dy 
$$ 
where $\mathcal{K}\colon \Omega\times \Omega\backslash \{(x,x)\colon x\in \Omega\}\to \R$ is symmetric and satisfies $C|x-y|^{-(n+2s)}\le \mathcal{K}(x,y)\le C'|x-y|^{-(n+2s)}$. Lower order perturbation can also be added: 
$$ 
\mathcal{L}u(x) = \mathcal{L}_\mathcal{K}u(x)  + \mathcal{L}_\mathcal{K'}u(x) + au(x) 
$$ 
where $\mathcal{K'}$ is as $\mathcal{K}$ with a $s'<s$, and $a$ is a $L^\infty$ function such that $\inf\,a>-\lambda_1$  where $\lambda_1$ is 1st eigenvalue of $\mathcal{L}_\mathcal{K}u(x)  + \mathcal{L}_\mathcal{K'}u(x)$ (so that $\mathcal{L}$ is positive). 

We can also consider local-non local operator like 
$$ 
\mathcal{L}u = -\Delta u + \mathcal{L}_\mathcal{K}u(x)  + a u  
$$ 
or 
$$ 
\mathcal{L}u = -\Delta u  +\Delta_J u  + a u 
$$ 
where 
$$  
\Delta_J u(x) = J*u(x)-u(x) = \int J(x-y)(u(y)-u(x))dy 
$$ 
for some smooth even non-negative function $J$ with compact support and $\int J=1$. 

The proofs for these extensions are easy adaptations of the ones presented here and the details are left to the interested reader.

\section*{Acknowledgements}

This work was partially supported by ANPCyT under grant PICT 2019-3837 and by CONICET under grant PIP 11220150100032CO. The authors are members of CONICET and are grateful for the support.

\end{document}